\documentclass[smallextended]{svjour3}       
\smartqed  
\usepackage{graphicx}
\usepackage{pdfsync}
\usepackage{latexsym,amsmath,amssymb}
\numberwithin{equation}{section}

\begin{document}

\title{Almost Yamabe Solitons on a Total Space of Almost Hermitian Submersions}

\titlerunning{Almost Yamabe Solitons on a Total Space of Almost Hermitian Submersions}        

\author{Tanveer Fatima \and Mehmet Akif Akyol \and Rakesh Kumar}


\institute{
Tanveer Fatima (Corresponding author) \at
Department of Mathematics and Statistics,\\
College of Sciences, Taibah University, Yanbu,\\
Kingdom of Saudi Arabia\\
\email{tansari@taibahu.edu.sa}
\and
Mehmet Akif Akyol \at
Department of Mathematics\\
Bingöl University, 12000 Bingöl, Turkey\\
\email {mehmetakifaykl@gmail.com}
\and
Rakesh Kumar \at
Department of Basic and Applied Sciences,\\
Punjabi University Patiala, 147 002, India\\
\email{dr$\_$rk37c@yahoo.co.in}
}

\date{Received: date / Accepted: date}

\maketitle

\begin{abstract}
This article presents the study of almost Hermitian submersion from an almost Yamabe soliton onto an almost Hermitian manifold.  A non-trivial example is also mentioned in order to guarantee the existence of such solitons on the total space of almost Hermitian submersions. We mainly focus on Kaehler submersions from Kaehler manifolds which are special case of almost Hermitian submersions. Under certain conditions, we find out that the fibres and the base manifold of such submersions are almost Yamabe soliton. We give the characterizations for an almost Yamabe soliton of a Kaehler submersion to be shrinking, steady and expanding in terms of extrinsic horizontal scalar curvature.  Moreover, we observe the behavior of torqued, recurrent and concurrent vector field of the total space of the Kaehler submersion. In particular, we obtain characterization for an almost Yamabe soliton consisting of concurrent vector fields. Meanwhile, we give some results of such submersions  when the total space is Yamabe soliton which is a particular case of almost Yamabe soliton.

\keywords{Kaehler submersions;  Almost Hermitian submersions;  Almost Yamabe Soliton; Torqued vector field; Concurrent vector field.}
\subclass{53C15; 53C40; 53C50}
\end{abstract}

\section{Introduction}

Hamilton \cite{a13} defined the notion of Yamabe flow in order to solve the Yamabe problem. The importance of the notion that Yamabe solitons seem as singularity models which are self-similar solutions for Yamabe flows. Many authors has studied this notion intensively. (see \cite{a3}, \cite{a4}, \cite{a5}, \cite{uc}, \cite{a15}, \cite{a17}.)
Almost Yamabe solitons, which are the natural generalization of Yamabe solitons was given as follows \cite{a2};

 {\it{A Riemannian manifold $M$ is said to be an almost Yamabe soliton $(M, g, \eta, \mu)$ if there exists a vector
 field $V$ on $M$ which satisfies
 \begin{align}
 	 \frac{1}{2} \mathcal{L}_{\eta} g=(\tau-\mu) g \label{ays},
 \end{align}
 where $\tau$ is the scalar curvature of $M,~ \mu$, a smooth function, $\eta$, a soliton field for $(M, g)$
 and $\mathcal{L}$ is the Lie-derivative.}} Moreover, we say that {\it{an almost Yamabe soliton is steady, expanding or shrinking if $\mu=0, ~\mu<0$ or $\mu>0$, respectively.}}
 An almost Yamabe soliton is said to be the Yamabe soliton if $\mu$ is a constant in the definition (\ref{ays}). It is obvious that Einstein manifolds are almost Yamabe solitons.

On the other hand, the theory of Riemannian submersions between Riemannian manifolds were introduced by O'Neill \cite{a16} and Gray \cite{a9}, respectively. It has been extended intensively in the last three decades. (For the interested readers, we refer to \cite{a8F}, \cite{a9}, \cite{a16}, \cite{a19}). Later, Riemannian submersions were considered between almost complex manifolds by Watson \cite{a20} under the name of almost Hermitian submersions. He showed that under almost Hermitian submersions the base manifold takes the form as total space in most cases. Since then, Riemannian submersions have been used as an effective tool to describe the structure of a Riemannian manifold equipped with a differentiable structure.

Considering the above studies of almost Yamabe solitons and Riemannian submersions, the objective of the present paper is to study the almost Hermitian submersions whose total space is an almost Yamabe soliton which is also verified through an example. Basic tools of the mentioned theory are given in section 2. Section 3 essentially concerns with Keahler submersions admitting almost Yamabe soliton. Some observations are also mentioned while the potential vector field of almost Yamabe soliton is torqued, recurrent or concurrent vector field.

\section{Preliminaries}
This section often involve well known definitions and concepts of almost Hermitian submersions which are present in subsequest sections and can be found through the index.
\\

Let $(M_1,g_1)$ and $(M_2,g_{\text{\tiny$2$}})$ be
Riemannian manifolds, where $\dim(M_1)$ is greater than $\dim(M_2)$. A surjective mapping
$\pi:(M_1,g_1)\rightarrow(M_2,g_{2})$ is called a
\emph{Riemannian submersion}
\cite{a16} if\\
\text{(S1)} $\pi$ has maximal rank, and \\
\text{(S2)} $\pi_{*}$, restricted to $\ker\pi_{*}^{\bot},$ is a linear
isometry.\\

In this case, for each $q\in M_2$, $\pi^{-1}(q)$ is a $k$-dimensional
submanifold of $M_1$ and called a \emph{fiber}, where $k=\dim(M_1)-\dim(M_2).$
A vector field on $M_1$ is called \emph{vertical} (resp.
\emph{horizontal}) if it is always tangent (resp. orthogonal) to
fibers. A vector field $X$ on $M_1$ is called \emph{basic} if $X$ is
horizontal and $\pi$-related to a vector field $X_{*}$ on $M_2,$ i.e.,
$\pi_{*}X_{p}=X_{*\pi(p)}$ for all $p\in M_1.$ We will denote by
$\mathcal{V}$ and $\mathcal{H}$ the projections on the vertical
distribution $\ker\pi_{*}$, and the horizontal distribution
$\ker\pi_{*}^{\bot},$ respectively. As usual, the manifold $(M_1,g_1)$ is called \emph{total manifold} and the manifold $(M_2,g_{2})$ is called \emph{base manifold} of the submersion $\pi:(M_1,g_1)\rightarrow(M_2,g_{2})$.
The geometry of Riemannian submersions is characterized by O'Neill's tensors $\mathcal{T}$ and
$\mathcal{A}$, defined as follows:
\begin{equation}\label{tuv}
\mathcal{T}_{U}{V}=\mathcal{V}\nabla_{\mathcal{V}{U}}\mathcal{H}{V}+\mathcal{H}\nabla_{\mathcal{V}{U}}\mathcal{V}{V},
\end{equation}
\begin{equation}\label{auv}
\mathcal{A}_{U}{V}=\mathcal{V}\nabla_{\mathcal{H}{U}}\mathcal{H}{ V}+\mathcal{H}\nabla_{\mathcal{H}{U}}\mathcal{V}{V}
\end{equation}
for any vector fields ${U}$ and ${V}$ on $M_1,$ where $\nabla$ is the
Levi-Civita connection of $g_1$. It is easy to see
that $\mathcal{T}_{{U}}$ and $\mathcal{A}_{{U}}$ are skew-symmetric
operators on the tangent bundle of $M_1$ reversing the vertical and
the horizontal distributions. We now summarize the properties of the
tensor fields $\mathcal{T}$ and $\mathcal{A}$. Let $V, W$ be  vertical
and $X, Y$ be horizontal vector fields on $M_1$, then we have
\begin{equation}\label{t1vw}
\mathcal{T}_{V}W=\mathcal{T}_{W}V,
\end{equation}
\begin{equation}\label{a1xy}
\mathcal{A}_{X}Y=-\mathcal{A}_{Y}X=\frac{1}{2}\mathcal{V}[X,Y].
\end{equation}
On the other hand, from (\ref{t1vw}) and (\ref{a1xy}), we obtain
\begin{equation}\label{nvw}
\nabla_{V}W=\mathcal{T}_{V}W+\hat{\nabla}_{V}W,
\end{equation}
\begin{equation}\label{nvx}
\nabla_{V}X=\mathcal{T}_{V}X+\mathcal{H}\nabla_{V}X,
\end{equation}
\begin{equation}\label{nxv}
\nabla_{X}V=\mathcal{A}_{X}V+\mathcal{V}\nabla_{X}V,
\end{equation}
\begin{equation}\label{nxy}
\nabla_{X}Y=\mathcal{H}\nabla_{X}Y+\mathcal{A}_{X}Y,
\end{equation}
where $\hat{\nabla}_{V}W=\mathcal{V}\nabla_{V}W$. If $X$ is basic
\[\mathcal{H}\nabla_{V}X=\mathcal{A}_{X}V.\]
\begin{remark}\label{remark1}
Throughout the article, we have assumed all horizontal vector fields as basic vector fields.
\end{remark}
It is not difficult to observe that $\mathcal{T}$ acts on the fibers
as the second fundamental form while $\mathcal{A}$  acts on the
horizontal distribution and measures of the obstruction to the
integrability of this distribution. For  details on Riemannian
submersions, we refer to O'Neill's paper \cite{a16} and
the book \cite{a8F}. Furthermore, the O'Neill tensors $\mathcal{T}$ and $\mathcal{A}$ satisfy
\begin{align*}	
	\sum_{j=1}^{r} g_1\left(\mathcal{T}_{U} U_{j}, \mathcal{T}_{V} U_{j}\right)=\sum_{i=1}^{n} g_1\left(\mathcal{T}_{U} X_{i}, \mathcal{T}_{V} X_{i}\right) \\
	\sum_{j=1}^{r} g_1\left(\mathcal{A}_{X} U_{j}, \mathcal{A}_{Y} U_{j}\right)=\sum_{i=1}^{n} g_1\left(\mathcal{A}_{X} X_{i}, \mathcal{A}_{Y} X_{i}\right) \\
	\sum_{j=1}^{r} g_1\left(\mathcal{A}_{X} U_{j}, \mathcal{T}_{U} U_{j}\right)=\sum_{i=1}^{n} g_1\left(\mathcal{A}_{X} X_{i}, \mathcal{T}_{U} X_{i}\right),
\end{align*}
where $\left\{U_{j}\right\}_{1 \leq j \leq r}$ and $\left\{X_{i}\right\}_{1 \leq i \leq n}$ are orthonormal frames of vertical distribution $\mathcal{V}$ and horizontal
distribution $\mathcal{H}$, respectively, for any $X, Y \in \mathcal{H}$ and $U, V \in \mathcal{V}$.

Denote the Riemannian curvature tensors of $M_1, M_2$ and any fiber by $R^{\prime}, R$
and $\hat{R}$, respectively. Then, we get

\begin{align}
	R(U, V, F, W)&=\hat{R}(U, V, F, W)+g_1\left(\mathcal{T}_{V} F, \mathcal{T}_{U} W\right)-g_1\left(\mathcal{T}_{U} F, \mathcal{T}_{V} W\right) \label{uvfw},\\
	R(X, Y, G, Z)&=R^{\prime}\left(X^{\prime}, Y^{\prime}, G^{\prime}, Z^{\prime}\right) \circ \pi-g_1\left(\mathcal{A}_{Y} G, \mathcal{A}_{X} Z\right)\notag \\
	&+2 g_1\left(\mathcal{A}_{X} Y, \mathcal{A}_{G} Z\right)+g_1\left(\mathcal{A}_{X} G,  \mathcal{A}_{Y} Z\right)\label{x,y,g,z},
\end{align}
for any horizontal vectors $X, Y, G, Z$ and vertical vectors $U, V, F, W$.
Moreover, denoting the mean curvature vector of any fiber by $H$ which is given as
$N=r H,$ such that
\begin{equation}
    N=\sum_{j=1}^{r} \mathcal{T}_{V_{j}} V_{j},\label{n12}
\end{equation}
where $\left\{V_{1}, V_{2}, \ldots, V_{r}\right\}$ is an orthonormal frame of $\mathcal{V}$. Also, remark that
$$
\mathcal{T}_{U} V=g(U, V) H
$$
is satisfied for $U, V$ tangent to $\mathcal{V}$ if and only if the fiber is totally umbilical submanifold. Indeed, the vector field $N$ is zero on $M_1$ if and only if any fiber is minimal.
Using \ref{n12}, one has $$
g_1\left(\nabla_{E} N, Z\right)=\sum_{j=1}^{r} g_1\left(\left(\nabla_{E} \mathcal{T}\right)\left(U_{j}, U_{j}\right), Z\right),
$$
where $Z$ is any horizontal vector field and $E$, an arbitarary vector field $M_1$.\\

A $(1,1)-$ tensor field $J$ on an $2 n-$ dimensional smooth manifold $M$ is said to be an almost complex structure if $J^{2}=-I$. An almost Hermitian manifold $(M, g, J)$ is a smooth manifold endowed with an almost complex structure $J$ and a Riemannian metric $g$ compatible in the sense that
$$
g\left(J X_{1}, X_{2}\right)+g\left(X_{1}, J X_{2}\right)=0, \quad X_{1}, X_{2} \in \chi(M).
$$
The Kaehler form of the almost Hermitian manifold is defined by $\Phi\left(X_{1}, X_{2}\right)=g\left(X_{1}, J X_{2}\right)$. An almost Hermitian manifold is called Kaehler manifold if $\nabla J=0,$ where $\nabla$ is a Levi-Civita connection with respect to $g$. 

Now, we recall some basic notions about almost Hermitian submersions from (\cite{a8F}, \cite{a13J}, \cite{a20}).

Let $\left(M_{1}, g_{1}, J_{1}\right)$ and $\left(M_{2}, g_{2}, J_{2}\right)$ be almost Hermitian manifolds. Riemannian submersion $\pi: M_{1} \rightarrow M_{2}$ is said an almost Hermitian submersion if $\phi$ is an almost complex map, i.e. if $\pi_{*} \circ J_{1}=J_{2} \circ \pi_{*}$ holds. Moreover, an almost Hermitian submersion is called a Kaehler if the total space is Kaehler manifold. It is well known that the fibres and the base manifold of an almost Hermitian submersions belong to the same class as the total space.

Finally, for a given almost Hermitian submersion the scalar curvatures of total space, the base space and the fibres are related by

\begin{align}	
	\tau=\left(\tau^{\prime} \circ \pi\right)+\hat{\tau}-\|\mathcal{T}\|^{2} \label{mtau}
\end{align}
such that $\|\mathcal{T}\|^{2}=\sum_{i, j} g_1\left(\mathcal{T}_{U_{j}} X_{i}, \mathcal{T}_{U_{j}} X_{i}\right),$
where $\hat{\tau}$ and $\tau^{\prime}$ are the scalar curvatures of any fiber of $\pi$ and the base manifold respectively \cite{a8F}.

\section{Almost Hermitian submersions from an almost Yamabe soliton}

This section includes our investigation on the almost Hermitian submersion, in particular, the Kaehler submersion $\pi:(M_1, g_1, J_1)\to( M_{2}, g_{2}, J_{2})$ whose total space admits an almost Yamabe soliton.

We start with an example to provide a formal assurance of the existence of such solitons on the total manifold of almost Hermitian submersions.

\begin{example}
Let $\left(\mathbb{R}^{4}, J_{1}, g_{1}\right)$ be an almost Hermitian  manifold endowed with an almost complex structure $\left(J_{1}, g_{1}\right)$ which is given by
	$$
	\begin{aligned}
		&g_{1}=e^{-2 x_{3}} d x_{1}^{2}+e^{-2 x_{4}} d x_{2}^{2}+d x_{3}^{2}+d x_{4}^{2} \\
		&J_{1}\left(x_{1}, x_{2}, x_{3}, x_{3}\right)=\left(-x_{2}, x_{1},-x_{4}, x_{3}\right)
	\end{aligned}
	$$
	Let $\left(\mathbb{R}^{2}, J_{2}, g_{2}\right)$ be an almost Hermitian manifold endowed with an almost complex structure $\left(J_{2}, g_{2}\right)$, given by
	$$
	\begin{aligned}
		&g_{2}=e^{2 y_{2}} d y_{1}^{2}+d y_{2}^{2} \\
		&J_{2}\left(y_{1}, y_{2}\right)=\left(-y_{2}, y_{1}\right)
	\end{aligned}
	$$
	Let $\pi:\left(\mathbb{R}^{4}, J_{1}, g_{1}\right) \longrightarrow\left(\mathbb{R}^{2}, J_{2}, g_{2}\right)$ be defined by
	$$
	\pi\left(x_{1}, x_{2}, x_{3}, x_{4}\right)=\left(\frac{x_{1}-x_{3}}{\sqrt{2}}, \frac{x_{2}-x_{4}}{\sqrt{2}}\right)
	$$
	Then
	$$
	\operatorname{Ker} \pi_{*}=\operatorname{span}\left\{U_{1}=e_{1}+e_{3}, \quad U_{2}=e_{2}+e_{4}\right\}
	$$
	and 
	$$\left(\operatorname{Ker} \pi_*\right)^\perp=\operatorname{span}\left\{X_{1}=e_{1}-e_{3}, \quad X_{2}=e_{2}-e_{4}\right\}$$
	where $\left\{e_{1}=e^{x_{3}} \partial x_{1}, e_{2}=e^{x _ 4} \partial x_{2}, \quad e_{3}=\partial x_{3}, e_{4}=d x_{4}\right\}$ is the basis of $T\mathbb{R}^4$.  Then
	$$
	\pi_{*} X_{1}=\sqrt{2} e_{1}^{*}, \quad \pi_{*} X_{2}=\sqrt{2} e_{2}^{*}
	$$
	where $\left\{e_{1}^{*}=e^{-y_{2}} \partial y_{1}, e_{2}^{*}=\partial y_{2}\right\}$ is the basis of $T \mathbb{R}^{2}$. Then, clearly
	$$
	\begin{aligned}
		&g_{1}\left(X_{1}, X_{1}\right)=2=g_{2}\left(\pi_{*} X_{1}, \pi_{*} X_{1}\right) \\
		&g_{1}\left(X_{2}, X_{2}\right)=2=g_{2}\left(\pi_{*}  X_{2}, \pi_{*}  X_{2}\right)
	\end{aligned}
	$$
	Therefore $\pi$ is a Riemannian Submersion from $\mathbb{R}^{4}$ to $\mathbb{R}^{2}$. Moreover
	$$
	\begin{aligned}
		J_{1} X_{1} &=-e_{3}+e_{4} \quad, \quad J_{1} X_{2}=e_{1}-e_{3} \quad \text { then } \\
		\pi_{*} J_{1} X_{1}=-\sqrt{2} e_{2}^{*} \quad, \quad \pi_{*} J_{1} X_{2}=\sqrt{2} e_{1}^{*}
	\end{aligned}
	$$
	Also,
	$$
	J_{2} \pi_{*} X_{1}=-\sqrt{2} e_{2}^{*} \quad, \quad J_{2} \pi_{*} X_{2}=\sqrt{2} e_{1}^{*}
	$$
	Hence
	$$
	\pi_{*} J_{1} X_{1}=J_{2} \pi_{*} X_{1} \quad \text { and } \quad \pi_{*} J_{1} X_{2}=J_{2} \pi_{*} X_{2}
	$$
	Thus $\pi$ is an almost Hermitian submersion from $\left(\mathbb{R}^{4}, J_{1}, g_{1}\right)$ to $\left(\mathbb{R}^{2}, J_{2}, g_{2}\right)$. Further
	$$
	g_{1 i j}=\left[\begin{array}{cccc}
		e^{-2 x_{3}} & 0 & 0 & 0 \\
		0 & e^{-2 x^{4}} & 0 & 0 \\
		0 & 0 & 1 & 0 \\
		0 & 0 & 0 & 1
	\end{array}\right] \text {then } g_{1}^{i j}=\left[\begin{array}{cccc}
		e^{2 x_{3}} & 0 & 0 & 0 \\
		0 & e^{2 x_{4}} & 0 & 0 \\
		0 & 0 & 1 & 0 \\
		0 & 0 & 0 & 1
	\end{array}\right]
	$$
$$	\begin{aligned}
		&\text { Then }\\
		&\Gamma_{11}^{1}=\Gamma_{11}^{2}=\Gamma_{11}^{4}=0 \quad, \quad \Gamma_{11}^{3}=e^{-2 x_{3}}\\
		&\Gamma_{12}^{1}=0=\Gamma_{21}^{1}, \quad \Gamma_{12}^{2}=0=\Gamma_{21}^{2}, \Gamma_{12}^{3}=0 \Gamma_{21}^{3}, \quad \Gamma_{12}^{4}=0= \Gamma_{21}^{4}\\
		&\Gamma_{13}^{1}=-1=\Gamma_{31}^{1}, \quad \Gamma_{13}^{2}=0=\Gamma_{31}^{2}, \quad \Gamma_{13}^{3}=0 =\Gamma_{31}^{3}, \Gamma_{13}^{4}=0=\Gamma_{31}^{4}\\
		&\Gamma_{14}^{1}=0=\Gamma_{41}, \quad \Gamma_{14}^{2}=0=\Gamma_{41}^{2}, \Gamma_{14}^{3}=0=\Gamma_{41}^{3}, \Gamma_{14}^{4}=0=\Gamma_{41}^{4}\\
		&\Gamma_{22}^{1}=\Gamma_{22}^{2}=\Gamma_{22}^{3}=0, \Gamma_{22}^{4}=e^{-2 x_{4}}\\
		&\Gamma_{23}^{1}=0=\Gamma_{32}, \quad{ }_{23}^{2}=0=\Gamma_{32}^{2}, \Gamma_{23}^{3}=0=\Gamma_{32}^{3}, \Gamma_{23}^{4}=0=\Gamma_{23}^{4}\\
		&\Gamma_{24}^{1}=-1=\Gamma_{42}^{1}, \quad \Gamma_{24}^{2}=0=\Gamma_{42}^{2}, \quad \Gamma_{24}^{3}=0=\Gamma_{42}^{3}, \quad \Gamma_{24}^{4}=0=\Gamma_{42}^{4}\\
		&\Gamma_{33}^{1}=\Gamma_{33}^{2}=\Gamma_{33}^{3}=\Gamma_{33}^{4}=0 \\
		&\Gamma_{44}^{1}=\Gamma_{44}^{2}=\Gamma_{44}^{3}=\Gamma_{44}^{4}=0
	\end{aligned}
	$$
	Using these Christoffel symbols, we obtain
	$$\operatorname{Ric}\left(e_{1}, e_{1}\right)=1, \quad \operatorname{Ric}\left(e_{2}, e_{2}\right)=1, \quad \operatorname{Ric}\left(e_{3}, e_{3}\right)=1 ~\text{and}~ \operatorname{Ric}\left(e_{4}, e_{4}\right)=1$$
	Hence the scalar curvature $\tau$ is
	$$
	\tau=\sum_{i} \operatorname{Ric}\left(e_{i}, e_{i}\right)=4
	$$
	We also have
	$$	\begin{aligned}
	\nabla_{U_{1}} U_{1}&=-X_{1}, \quad  &\nabla_{U_{1}} U_{2}&=0, \quad  &\nabla_{U_{1}} X_{1}&=U_{1}, \quad &\nabla_{U_{1}} X_{2}&=0\\
	\nabla_{U_{2}} U_{1}&=0, \quad  &\nabla_{U_{2}} U_{2}&=-X_{2}, \quad  &\nabla_{U_{2}} X_{1}&=0, \quad &\nabla_{U_{2}} X_{2}&=U_{2}\\
	\nabla_{X_{1}} U_{1}&=-X_{1}, \quad  &\nabla_{X_{1}} U_{2}&=0, \quad  &\nabla_{X_{1}} X_{1}&=U_{1} , \quad &\nabla_{X_{1}} X_{2}&=0\\
	\nabla_{X_{2}} U_{1}&=0, \quad  &\nabla_{X_{2}} U_{2}&=-X_{2}, \quad  &\nabla_{X_{2}} X_{1}&=0, \quad &\nabla_{X_{2}} X_{2}&=U_{2}
		\end{aligned}
$$
	Since $T \mathbb{R}^{4}=\left(\operatorname{Ker} \pi_{*}\right)\oplus \left(\operatorname{Ker} \pi_{*}\right)^{\perp}$, therefore any vectors $Z_{1}, Z_{2}, Z_{3} \in T \mathbb{R}^{4}$ can be written as
	$$
	\begin{aligned}
		&Z_{1}=a_{1} U_{1}+b_{1} U_{2}+c_{1} X_{1}+d_{1} X_{2} \\
		&Z_{2}=a_{2} U_{1}+b_{2} U_{2}+c_{2} X_{1}+d_{2} X_{2} \\
		&Z_{3}=a_{3} U_{1}+b_{3} U_{2}+c_{3} X{1}+d_{3} X_{2}
	\end{aligned}
	$$
	where $a_{i}, b_{i}, c_{i}, d_{i} \in \mathbb{R}$, then
	\begin{equation}\label{eq1}
	g_{1}\left(Z_{2}, Z_{3}\right)=2\left(a_{2} a_{3}+b_{2} b_{3}+c_{2} c_{3}+d_{2} d_{3}\right)
	\end{equation}
	We know that
	\begin{align}\nonumber
		\frac{1}{2}\left(\mathcal{L}_{Z_1}g_{1}\right)\left(Z_{2}, Z_{3}\right)&=\frac{1}{2}\left[g_{1}\left(\nabla_{Z_{2}} Z_{1}, Z_{3}\right)+g_{1}\left(\nabla_{Z_{3}} Z_{1}, Z_{2}\right)\right] \\ \nonumber
		&=c_{1} a_{2}\left(2 a_{3}-c_{3}\right)+a_{1} b_{2}\left(2 b_{3}+d_{3}\right)-a_{1} c_{2}\left(2 c_{3}+a_{3}\right) \\ \label{eq2}
		\quad ~&-b_{1} d_{2}\left(2 d_{3}+b_{3}\right)-a_{1} a_{2} c_{3}-b_{1} b_{2} d_{3}+c_{1} c_{2} a_{3}+d_{1} d_{2} b_{3}
	\end{align}
	Using (\ref{eq1}) and (\ref{eq2}), the following expression
	$$
	\frac{1}{2}\left(\mathcal{L} _{Z_{1}} g_{1}\right)\left(Z_{2}, Z_{3}\right)=(\tau-\lambda) g_{1}\left(Z_{2}, Z_{3}\right)
	$$
	becomes
	$$\begin{aligned}
		& c_{1} a_{2}\left(2 a_{3}-c_{3}\right)+d_{1} b_{2}\left(2 b_{3}+d_{3}\right)-a_{1} c_{2}\left(2 c_{3}+a_{3}\right)-b_{1} d_{2}\left(2 d_{3}+b_{3}\right) \\
		& -a_{1} a_{2} c_{3}-b_{1} b_{2} d_{3}+c_{1} c_{2} a_{3}+d_{1} d_{2} b_{3}=(4-\lambda) 2\left(a_{2} a_{3}+b_{2} b_{3}+c_{2} c_{3}+d_{2} d_{3}\right)
	\end{aligned}$$
		
	Since $a_{i}, b_{i}, c_{i}, d_{i} \in \mathbb{R}$, hence $\lambda$ is constant and $\left(\mathbb{R}^{4}, J_{1}, g_{1}\right)$ becomes Yamabe soliton. Moreover, for some suitable choices of $a_{i}, b_{i}, c_{i}$ and $d_{i}$, the Yamabe soliton $\left(\mathbb{R}^{4}, J_{1}, g_{1}\right)$ will be shrinking, steady or expanding according to $\lambda<0, \lambda=0$ or $\lambda>0$, respectively.

\end{example}

Now we mention some theorems and corollaries for any fibre of the submersion and the base space to be an almost Yamabe soliton and a Yamabe soliton.

First we recall the following:

\begin{lemma}\label{lfm}\cite{a8} Following are the statements which are equivalent to each other for a Riemannian submersion $\pi$ from a Riemannian manifold $(M_1,g_1)$ onto a Riemannian manifold $(M_2,g_2)$
\begin{itemize}
    \item [(i)] The horizontal distribution $\mathcal H$ is parallel with respect to $\nabla$ on $M$.
    \item [(ii)] The vertical distributiuin $\mathcal V$ is parallel with respect to $\nabla$ on $M$.
    \item [(iii)] The fundamental tensors $\mathcal T$ and $\mathcal A$ vanishes identically.\\
\end{itemize} 
\end{lemma}

\begin{theorem} \label{T1}
	Let $(M_1, g_1, J_1)$ be an almost Yamabe soliton with the potential vector field  $\eta$ and $\pi:(M_1, g_1, J_1)\to( M_{2}, g_{2}, J_{2})$
	be a Kaehler  submersion. Then the totally geodesic fibres of the submersion are almost Yamabe solitons if the potential vector field $\eta$ is vertical.
\end{theorem}
\begin{proof} 
	Since $(M_1, g_1, J_1$ be an almost Yamabe soliton, then from (\ref{ays}), we have
	\begin{align*}
		\frac{1}{2}(\mathcal{L}_{\eta} g_{1})(U, V)&=(\tau-\mu) g_{1}(U, V)\\
		\frac{1}{2}\{g_{1}(\nabla^1_U{\eta}, V)+g_{1}(\nabla_{V}^{1} \eta, U)\}&=(\tau-\mu) g_{1}(U, V).
	\end{align*}	
for any vertical vector fields $U, V$ on $M_1$. Since the potential vector field $\eta$ is vertical, by using (\ref{nvw}) we get,\\
	
	$\dfrac{1}{2}\{g_{1}(\hat{\nabla}_{U} \eta, V)+g_{1}(\hat{\nabla}_{V} \eta, U)\}=(\tau-\mu ) g_{1}\left(U,V\right).$\\
	By using (\ref{mtau}) and the assumption that the fibres of the submersion are totally geodesic i.e., the second fundamental form of fibres $\mathcal T=0$
	\[\dfrac{1}{2}\{g_{1}(\hat{\nabla}_{U} \eta, V)+g_{1}(\hat{\nabla}_{V} \eta, U)\}=({\hat\tau}-\mu){g}_1(U,V).\]
	In view of the relation (\ref{ays}), above equation yields that 
		\[\dfrac{1}{2}(\mathcal{L}_{\eta} {g}_{1}) (U,V)=({\hat\tau}-\mu){g}_1(U,V),\]
which shows that the fibers of the submersion $\pi$ are almost Yamabe solitons with the potential vector field $\eta$.\\
\end{proof}

We remark that an almost Yamabe soliton is said to be Yamabe soliton if $\mu$ is constant.
	
As a consequence of the above theorem, we have  

\begin{corollary}\label{C1}
	Let $\pi$ be a Kaehler submersion from a Kaehler  manifold $\left(M_1, g_{1}, J_{1}\right)$ admitting
	an almost Yamabe soliton to a Kaehler manifold $\left(M_{2}, g_{2}, J_{2}\right)$. Then the totally geodesic fibres of the submersion are Yamabe soliton if the potential vector field is vertical.	
\end{corollary}

\begin{theorem}\label{T2}
	Let $\pi:(M_1, g_1, J_1)\to( M_{2}, g_{2}, J_{2})$ be a Kaehler submersion admitting an almost Yamabe soliton with the potential vector field $\eta$ and a smooth function $\mu$ such that the potential vector field is horizontal. If the second fundamental form of fibres $\mathcal {T}$ vanishes then the base manifold $(M_2,g_2,J_2)$ is also an almost Yamabe soliton.
\end{theorem}

\begin{proof}  Since $(M_1,g_1,J_1)$ is an almost Yamabe soliton, then from (\ref{ays}), we get
$$	\frac{1}{2} \{ g_{1}\left(\nabla_{X} \eta, Y\right)+g_{1}\left(\nabla_{Y} \eta, X\right)\}=(\tau-\mu) g_{1}\left(X, Y\right)$$
for any horizontal vector fields $X$ and $Y$ on $M_1$. By simply using (\ref{nxy}) and (\ref{mtau}) above equation reaches to
$\frac{1}{2}\left\{g_{1}\left(\mathcal{H}\left(\nabla_{X} \eta\right), Y\right)+g(\mathcal{H}(\nabla_Y \eta), X)\right\}=\{(\tau^\prime\circ \pi)+\|\mathcal {T}\|^2-\hat{\tau}-\mu\}g(X,Y)$.\\

Since $\mathcal T=0$ and also we note  that $\mathcal{H}\left(\nabla_{X}\eta\right)$ and $\mathcal{H}(\nabla_Y\eta)$ are $\pi$ -related to ${\nabla}^{2}_{X_*}\eta_*$ and ${\nabla}^{2}_{Y_*}\eta_*$, respectively. It follows that
\begin{equation}\label{E2}
	\frac{1}{2}\left\{g_{2}(\nabla_{X_*}^2 \eta_*, Y_*)+g_2(\nabla_{Y_*}^2 \eta_*, X_*)\right\}=\{(\tau^\prime\circ \pi_*)-\mu\}g_2(X_*,Y_*)\circ \pi_*,
\end{equation}
where the vector fields $X_*, Y_*$ are tangent to $(M_2,g_2,J_2)$ and $\eta_*$ is potential vector field on $M_2$, which is $\pi$-related to $eta$ on $M_1$.  By using (\ref{ays}) in (\ref{E2}), it shows that the base manifold $(M_2,g_2,J_2)$ is almost Yamabe soliton with the potential vector field $\eta_*$ and a function $\mu$. 
\end{proof}
As a conclusion of the above theorem, we state the following corollary;  
\begin{corollary}\label{C2}
	Let $\pi:(M_1, g_1, J_1)\to( M_{2}, g_{2}, J_{2})$ be a Kaehler submersion admitting Yamabe soliton with the potential vector field $\eta$ and a smooth function $\mu$ such that the potential vector field is horizontal. If the second fundamental form of fibres $\mathcal {T}$ vanishes then the base manifold $(M_2,g_2,J_2)$ is also a Yamabe soliton.
\end{corollary}

\section{Characterizations on total space of Almost Hermitian submersions}
This section gives the characterization on almost Yamabe soliton by investigating the the relation between extrinsic horizontal curvature and extrinsic vertical curvature with function $\mu$. We also obtain some results on torqued, recurrent and concurrent vector fields on Kaehler submersion.
\begin{theorem}\label{T3} 
	Let $\pi$ be a Kaehler submersion admitting an almost Yamabe soliton
	$(M_1, g_1, \eta, \mu)$ such that the potential vector field $\eta$ is vertical and the horizontal distribution is parallel. Then, the extrinsic horizontal scalar curvature $\left.\tau\right|_{\mathcal{H}}$ satisfies
	$$
	\left.\tau\right|_{\mathcal{H}}-\mu=0.
	$$
\end{theorem}
\begin{proof} Since the total space is an almost Yamabe soliton, then from (\ref{ays}), we get
$$
\frac{1}{2}\left\{g_1\left(\nabla_{X} \eta, Y\right)+g_1\left(\nabla_{Y} \eta, X\right)\right\}=(\tau-\mu) g_1(X, Y)
$$
for any horizontal vectors $X, Y$ on $M_1$. Considering (\ref{nxy}) in the above equation, we have
$$
\frac{1}{2}\left\{g_1\left(\mathcal{A}_{X} \eta, Y\right)+g_1\left(\mathcal{A}_{Y} \eta,X\right)\right\}=(\tau|_{\mathcal{H}}-\mu) g_1(X, Y).
$$
Since the horizontal distribution is parallel, therefore in view of the properties of the tensor field $\mathcal{A}$ above equality in turn yields 
$$
\left(\left.\tau\right|_{\mathcal{H}}-\mu\right) g_1(X, Y)=0,
$$
for any horizontal vector fields $X, Y$ on $M_1$, which proves our claim.
\end{proof}

As an immediate consequence of Theorem \ref{T3} we conclude the following observations;

\begin{corollary}\label{C2} 
	Let $\pi:(M_1, g_1, J_1)\to( M_{2}, g_{2}, J_{2})$ be a Kaehler submersion admitting an almost Yamabe soliton $(M_1,g,\eta,\mu)$ such that $\eta$ is vertical. If the horizontal distribution is parallel then the followings are hold;
\begin{itemize}
\item[(i)] ($M_1,g_1,\eta,\mu)$ is shrinking if and only if the extrinsic horizontal scalar curvature is positive.
\item[(ii)] ($M_1,g_1,\eta,\mu)$ is expanding if and only if the extrinsic horizontal scalar curvature is negative.
\item[(iii)]  $(M_1,g_1,\eta,\mu)$ is steady if and only if the extrinsic horizontal scalar curvature is zero.
\end{itemize}
\end{corollary}

\begin{theorem}\label{T4} 
	Let $\pi$ be a Kaehler submersion admitting an almost Yamabe soliton
	$(M_1, g_1, \eta, \mu)$ such that the potential vector field $\eta$ is horizontal and the horizontal distribution is parallel. Then, the extrinsic vertical scalar curvature $\hat{\tau}$ satisfies $$\hat{\tau}-\mu=0.$$
\end{theorem}
\begin{proof} Since $(M_1,g_1,J_1)$ is an almost Yamabe soliton and the potential vector field $\eta$ is horizontal vector field then by using (\ref{nvx}), we  arrive at,
$$\frac{1}{2}\{g_1(\mathcal{T}_U\eta,V)+g_1(\mathcal{T}_V\eta,U)\}=(\hat{\tau}-\mu)g(U,V)$$
for any vertical vector field U, V on $M$. Since the horizontal distribution is parallel, by using Lemma \ref{lfm}, the extrinsic vertical scalar curvature satisfies $(\hat{\tau}-\mu)=0.$\\
\end{proof}

In view of the relation of extrinsic vertical scalar curvature $\hat\tau$ with $\mu$ we mention the following conclusion;
\begin{corollary}\label{C3}
	Let $\pi$ be a Kaehler submersion admitting an almost Yamabe soliton
	$(M_1, g_1, \eta, \mu)$ such that the potential vector field $\eta$ is horizontal and the horizontal distribution is parallel. Then,
	\begin{itemize}
	\item[(i)] $(M_1,g_1,\eta, \mu)$ is shrinking if and only if the fibre has  positive scalar curvature.

	\item[(ii)] $(M_1,g_1,\eta, \mu)$ is expanding if and only if the fibre has  negative scalar curvature.
	
	\item[(iii)] $(M_1,g_1,\eta, \mu)$ is steady if and only if the fibre has zero scalar curvature.
	\end{itemize}
	
\end{corollary}

\begin{definition}\label{D1} 
	A nowhere vanishing vector field $\eta$ on a Riemannian manifold $M$ is called torse-forming if
	\begin{align}
	\nabla_{W}\eta=\varphi W+\alpha(W) \eta, \label{tv}
	\end{align}
	where $\varphi$ is a function, $\alpha$ is a 1-form. The vector field $\eta$ is called
	concircular, if the 1-form $\alpha$ vanishes identically.
	The vector field $\eta$ is called concurrent if the 1 -form $\alpha$ vanishes
	identically and the function $\varphi=1$. The vector field $\eta$ is called 
	recurrent if the function $\varphi=0$ \cite{a5.1}, \cite{a5.2}. 

	Finally if $\varphi=\alpha=0$, then the vector field $\eta$ is called a parallel vector field. The nowhere zero vector field $\eta$ is called a torqued vector field if it satisfies
		\begin{align}
	\nabla_{W} \eta=\varphi W+\alpha(W) \eta, \quad \alpha(\eta)=0, \label{tv1}
		\end{align}
	where the function $\varphi$ is called the torqued function and 1-form $\alpha$ is called the torqued form of $\eta ;$ see $[8]$.
	
\end{definition}

\begin{theorem}\label{T5}
	Let $\pi:\left(M_{1}, g_{1}, J_{1}\right) \rightarrow\left(M_{2}, g_{2}, J_{2}\right)$ be a Kaehler submersion. Then, 
	\begin{itemize}
\item[(i)] For any torqued vector field $\eta$ on $\Gamma\left(\operatorname{ker} \pi_{*}\right)^{\perp}, J_{1} \eta$ is not a torqued vector field on $\Gamma{(ker \pi_*)}^\bot .$
\item[(ii)] For torqued vector field $\eta \in \Gamma\left(k e r \pi_{*}\right)^{\perp},[\eta, J \eta]$ does not vanish.
\end{itemize}
\end{theorem}

\begin{proof} Since $M_{1}$ is a Kaehler manifold, then we have
$$
\nabla_{X}^{1} J_{1} \eta=J_{1} \nabla_{X}^{1} \eta,~ \text{for~any} \quad X \in \Gamma\left(k e r \pi_{*}\right)^{\perp},$$
where $\eta$ is a torqued vector field which belongs to $\Gamma\left(\operatorname{ker} \pi_{*}\right)^{\perp}$ and $\nabla^1$ is a Levi-Civita connection on $M_1.$ 
Using (\ref{nxy}), we have
$$
\nabla_{X}^{1} J_{1} \eta=J_{1}\left(\mathcal{H} \nabla_{X}^{1} \eta+\mathcal{A}_{X} \eta\right).
$$
From (\ref{tv1}) and $\mathcal{A}=0$ we arrive at
\begin{align}
\nabla_{X}^{1} J_{1} \eta=J_{1}\left(\mathcal{H} \nabla_{X}^{1} \eta\right)=\varphi J_1 X+\alpha(X)J_1\eta \label{nxjn}
\end{align}
which implies that $J_{1} \eta$ is never a torqued vector field. We also have
\begin{align}
\mathcal{H} \nabla_{J_{1} \eta}^{1} \eta=\varphi J_1 \eta+\alpha(J_1\eta)\eta. \label{nxjn1}
\end{align}
Now, by taking $\eta$ instead of $X$ in (\ref{nxjn}) and taking the difference with (\ref{nxjn1}), we obtain $(ii)$.
\end{proof}

In particular, if we consider $\eta$ as a recurrent or concurrent vector field then we have the following corollaries;

\begin{corollary}\label{C4}
	Let $\pi:\left(M_{1}, g_{1}, J_{1}\right) \rightarrow\left(M_{2}, g_{2}, J_{2}\right)$ be a Kähler submersion. Then, 
	\begin{itemize}
\item[(i)] For a recurrent vector field $\eta$ on $\Gamma\left(\operatorname{ker} \pi_{*}\right)^{\perp}, J_{1} \eta$ is also a recurrent vector field on $\left(\operatorname{ker} \pi_{*}\right)^{\perp} .$
\item[(ii)] For recurrent vector field $\eta \in \Gamma\left(k e r \pi_{*}\right)^{\perp}, [\eta, J \eta]$ vanishes.
	\end{itemize}
\end{corollary}

For the concurrent vector field, we have the following result from \cite{a12}:
\begin{corollary}\label{C5}
	Let $\pi:\left(M_{1}, g_{1}, J_{1}\right) \rightarrow\left(M_{2}, g_{2}, J_{2}\right)$ be a Kähler submersion. Then, 
	\begin{itemize}
	\item[(i)] For a concurrent vector field $\eta$ on $\Gamma\left(\operatorname{ker} \pi_{*}\right)^{\perp}, J_{1} \eta$ is not a concurrent vector field on $\left(\operatorname{ker} \pi_{*}\right)^{\perp} .$
\item[(ii)] For concurrent vector field $\eta \in \Gamma\left(k e r \pi_{*}\right)^{\perp},[\eta, J \eta]$ vanishes.
\end{itemize}
\end{corollary}

\begin{theorem}\label{T6} Let $\pi:\left(M_{1}, g_{1}, J_{1}\right) \rightarrow\left(M_{2}, g_{2}, J_{2}\right)$ be
	a Kähler submersion and $\eta$ be the concurrent potential vector field and $\mu$, a function on $M$. If the second fundamental form of the fibres are totally geodesic then the extrinsic vertical scalar curvature $\hat\tau$ is equal to $\mu+1.$ for $\eta\in \Gamma{(kaer \pi_*)}$
\end{theorem}
\begin{proof} As $\left(M_{1}, g_{1}, J_{1}\right)$ be an almost Yamabe soliton with potential vector field $\eta$, then for any $U,V\in\Gamma(ker \pi_*)$ we have 

	\[\frac{1}{2}\{g_{1}(\nabla^1_U{\eta}, V)+g_{1}(\nabla_{V}^{1} \eta, U)\}=(\tau-\mu) g_{1}(U, V)\]
	Since $\mathcal T=0$, using (\ref{mtau}) we get
\[	\frac{1}{2}\{\hat g(\hat{\nabla}_{U} \eta, V)+\hat g (\hat{\nabla}_{V} \eta, U)\}=(\hat{\tau}-\mu)\hat g(U,V).\]
Since the potential vector field $\eta$ is concurrent vector field on $M_1$, with the help of (\ref{tv}) we conclude the theorem.
\end{proof}

In view of the relation between the extrinsic vertical scalar curvature $\hat\tau$ and $\mu$, we give the following result;\\

\begin{corollary}\label{C6}
	Let $\pi:\left(M_{1}, g_{1}, J_{1}\right) \rightarrow\left(M_{2}, g_{2}, J_{2}\right)$ be a Kähler submersion and $(M_{1}, g_{1})$ be an almost Yamabe soliton with the concurrent potential vector field $\eta$. If $\eta\in\Gamma(ker \pi_*)$ then we have
	\begin{itemize}
\item[(i)] $(M_1,g_1,\eta, \mu)$ is shrinking if and only if the extrinsic vertical scalar curvature $\hat\tau>1$.
	
\item[(ii)] $(M_1,g_1,\eta, \mu)$ is expanding if and only if  the extrinsic vertical scalar curvature $\hat\tau<1$.
	
	\item[(iii)] $(M_1,g_1,\eta, \mu)$ is steady if and only if  the extrinsic vertical scalar curvature $\hat\tau= 1$.\\
\end{itemize}
\end{corollary}

\subsection*{\bf Data Availability Statement}
 Data sharing not applicable to this article as no datasets were generated or analysed during the current study.
 

\end{document}